\documentclass[11pt]{article}

\usepackage[margin=1in]{geometry}         

\usepackage{longtable}
\usepackage{graphicx}
\usepackage{caption,subcaption}
\usepackage{stmaryrd}
\usepackage{framed}
\usepackage{amssymb}
\usepackage{epstopdf}
\usepackage{amsmath,amsfonts,amssymb}
\usepackage{enumerate}
\usepackage{multirow}
\usepackage{fixltx2e}
\usepackage{bbm}
\usepackage{url}
\usepackage{cite}
\usepackage{fullpage}
\usepackage{fancyhdr}
\usepackage{empheq}
\usepackage{color}
\usepackage{float}
\usepackage{soul}
\usepackage{graphicx}
\usepackage[doc]{optional}
\usepackage{enumitem}
\usepackage{xcolor}
\usepackage[toc,page]{appendix}
\usepackage{theorem}
\usepackage{array}
\usepackage{booktabs}
\usepackage{epsfig}
\usepackage{latexsym}

\definecolor{labelkey}{rgb}{0,0.08,0.45}
\definecolor{refkey}{rgb}{0,0.6,0.0}

\definecolor{myblue}{rgb}{.9, .9, 1}

\newtheorem{theorem}{Theorem}[section]
\newtheorem{lemma}[theorem]{Lemma}
\newtheorem{corollary}[theorem]{Corollary}
\newtheorem{proposition}[theorem]{Proposition}
\newtheorem{definition}[theorem]{Definition}

\theoremstyle{plain}{\theorembodyfont{\rmfamily}

\theoremstyle{plain}{\theorembodyfont{\rmfamily}
}
\theoremstyle{plain}{\theorembodyfont{\rmfamily}
}
\theoremstyle{plain}{\theorembodyfont{\rmfamily}
}
\theoremstyle{plain}{\theorembodyfont{\rmfamily}
}

\theoremstyle{plain}{\theorembodyfont{\rmfamily}
}

\theoremstyle{plain}{\theorembodyfont{\rmfamily}
\newtheorem{customC}{Algorithm}
\newenvironment{Calg}[1]
  {\renewcommand\thecustomC{#1}\customC}
  {\endcustomC}

\theoremstyle{plain}{\theorembodyfont{\rmfamily}

\theoremstyle{plain}{\theorembodyfont{\rmfamily}

\newcommand{\dom}{\ensuremath{\operatorname{dom}}}

\def\RR{{\mathbb{R}}}

\def\NN{{\mathbb{N}}}

\newcommand{\la}{\langle}
\newcommand{\ra}{\rangle}

\newcommand{\nexto}{\kern -0.54em}

\newcommand{\dZ}{{\cal Z \kern -0.7em Z}}
\newcommand{\dC}{{\rm\hbox{C \kern-0.8em\raise0.2ex\hbox{\vrule height5.4pt width0.7pt}}}}
\newcommand{\dQ}{{\rm\hbox{Q \kern-0.85em\raise0.25ex\hbox{\vrule height5.4pt width0.7pt}}}}

\newenvironment{retraitsimple}{\begin{list}{--~}{
 \topsep=0.3ex \itemsep=0.3ex \labelsep=0em \parsep=0em
 \listparindent=1em \itemindent=0em
 \settowidth{\labelwidth}{--~} \leftmargin=\labelwidth
}}{\end{list}}

\begin{document}
\title{A hybrid method for solving systems of operator inclusion problems}
\author{R. D\'iaz Mill\'an\footnote{ Federal Institute of Education, Science and
Technology, Goi\^ania, Brazil, e-mail: rdiazmillan@gmail.com}} \maketitle
\begin{abstract}
In this paper, we propose an algorithm combining the forward-backward splitting method and the alternative projection method for solving the system of splitting  inclusion problem. We want to find a point in the interception of a finite number of sets  that we don't know, the solution of each component of the system. The algorithm consists of approximate the sets involved in the problem by separates halfspaces which are a known strategy.   By finding these halfspaces in each iteration we use only one inclusion problem of the system. The iterations consist of two parts, the first contains an explicit Armijo-type search in the spirit of the extragradient-like methods for variational inequalities.  In the iterative process, the operator forward-backward is computed only one time for each inclusion problem, this represents a great computational saving because the computational cost of this operator is nothing cheap. The second part consists of special projection step, projecting in the separating halfspace. The convergence analysis of the proposed scheme is given assuming monotonicity all operators, without any Lipschitz continuity assumption.

\bigskip

\noindent{\bf Keywords:} Armijo-type search,
Maximal monotone operators, Forward-Backward, Alternative projection, Systems of inclusion problems, Armijo-type search

\noindent{\bf Mathematical Subject Classification (2008):} 90C47,
49J35.
\end{abstract}
\section{Introduction}
The goal of this paper is to present an algorithm for solving the system of inclusion problem, in which each component of the system is a sum of two operators, one point-to-set and the other point-to-point. Given a finite family of pair of operators $\{ A_i,B_i \}_{i \in \mathbb{I}}$, with $\mathbb{I}=:(1,2,\cdots, m)$ and $m\in \NN$.
The system of inclusion problem consists in:
\begin{equation}\label{problema}
\mbox{find} \ \ x^*\in \RR^n  \ \  \mbox{such that} \ \  0\in A_i(x^*)+B_i(x^*) \ \ \mbox{for all} \ \ i\in \mathbb{I},
\end{equation}
where, for all $i\in\mathbb{I}$, the operators $A_i:\dom(A_i)\subset \RR^n \rightarrow \RR^n$ are point-to-point and maximal monotone and the operators $B_i:\dom(B_i)\subset \RR^n\rightarrow 2^{\RR^n}$ are point-to-set maximal monotone operators. The solution of the problem, denoted by $S_*$, is given by the interception of the solution of each component of the system, i.e., $S_*=\cap_{i\in \mathbb{I}} S^i_*$, where $S^i_*$ is defined as $S^i_*:=\{x\in \RR^n: 0\in A_i(x)+B_i(x)\}$.

 Many problems in mathematics and science in general can be modeled as problem \eqref{problema}, for example, taking the operators $B_i = N_{C_i}$ with $C_i \subset \RR^n$ convex sets for all $i\in \mathbb{I}$ we have the system of variational inequalities, introduced by I.V. Konnov in \cite{konnov1}, which have been studied in \cite{konnov, eslam, gibali, gibali2, gibali3, konnov1} and others. Some forward-backward algorithms for solving the inclusion problem, when the system contains just one equation, the hypothesis of Lipschitz continuity is very common see \cite{tseng, ranf}. In this paper, we improve this results assuming only maximal monotonicity  for all operators $A_i$ and $B_i$. Also, we improve the linesearch proposed by Tseng in \cite{tseng}, calculating only one time the forward-backward operator in each tentative to find the step size. Another advantage of the proposed algorithm is that in each iteration we not calculate the interception of any hyperplane like was do it in \cite{reinier}, and we use only one component of the system in each step of the algorithm, in the spirits of the alternative projection method. This improves the algorithm in the computational sense because any hard subproblem must be solved and because the forward-backward operator is very expensive to compute. The present work follows the ideas of the works \cite{rei-yun, phdthesis, borw-baus}.

Problem \eqref{problema} have many applications in operations research, optimal control, mathematical physics, optimization and differential equations. This kind of problem has been deeply studied and has recently received a lot of attention, due to the fact that many nonlinear problems, arising within applied
areas are mathematically modeled as nonlinear operator system of equations and/or inclusions, which each one is decomposed as a sum of two operators.
\section{Preliminaries}
 In this section, we present some notation, definitions and results needed for the
convergence analysis of the proposed algorithm.  The inner product in
$\RR^n$ is denoted by $\la \cdot , \cdot \ra$ and the norm induced
by the inner product by $\|\cdot\|$. We denote by $2^{C}$ the power set of $C$. For $X$ a nonempty, convex and
closed subset of $\RR^n$, we define the orthogonal projection of $x$
onto $X$ by $P_X(x)$, as the unique point in $X$, such that $\|
P_X(x)-x\| \le \|y-x\|$ for all $y\in X$. Let $N_X(x)$
be the normal cone to $X$ at $x\in X$, i.e.,
$N_X(x):=\{d\in \RR^n\, : \, \la d,x-y\ra\ge 0 \;\; \forall
y\in X\}$. Recall that an operator $T:\RR^n \rightarrow 2^{\RR^n}$
is monotone if, for all $(x,u),(y,v)\in Gr(T):=\{(x,u)\in
\RR^n\times \RR^n : u\in T(x)\}$, we have $\la x-y, u-v \ra \ge 0,$
and it is maximal if $T$ has no proper monotone extension in the
graph inclusion sense. Now some known results.

\begin{proposition}\label{proj}
Let $X$ be any nonempty, closed and convex set in $\RR^n$. For all $x,y\in \RR^n$ and all $z\in X $ the following hold:
\begin{enumerate}
\item $ \|P_{X}(x)-P_{X}(y)\|^2 \leq \|x-y\|^2-\|(P_{X}(x)-x)-\big(P_{X}(y)-y\big)\|^2.$
\item $\la x-P_X(x),z-P_X(x)\ra \leq 0.$
\item $P_X=(I+N_X)^{-1}.$
\end{enumerate}
\end{proposition}

\begin{proof}
 (i) and (ii) see    Lemma    $1.1$    and    $1.2$    in    \cite{zarantonelo}. (iii) See Proposition $2.3$ in \cite{bauch}.
\end{proof}

 In the following we state some useful results on maximal monotone operators.
\begin{lemma}\label{bound}
Let $T:dom(T)\subseteq \RR^n\rightarrow 2^{\RR^n}$ be a maximal monotone operator.
Then,
\begin{enumerate}
\item $Gr(T)$ is closed.
\item $T$ is bounded on bounded subsets of the interior of
its domain.
\end{enumerate}
\end{lemma}

\begin{proof}
\begin{enumerate}
\item See Proposition $4.2.1$(ii) in \cite{iusem-regina}.
\item Consequence of Theorem 4.6.1(ii) of in \cite{iusem-regina}.
\end{enumerate}
\end{proof}

\begin{proposition}\label{inversa}
Let $T:dom(T)\subseteq\RR^n \rightarrow 2^{\RR^n}$ be a point-to-set and maximal monotone operator. Given $\beta >0$ then the operator $(I+\beta\, T)^{-1}: \RR^n \rightarrow dom(T)$ is single valued and maximal monotone.
\end{proposition}
\begin{proof}
See Theorem $4$ in \cite{minty}.
\end{proof}
\begin{proposition}\label{parada}
Given $\beta>0$ and $A: dom(A)\subseteq \RR^n\to \RR^n$ be a monotone operator and $B: dom(B)\subseteq \RR^n\rightarrow 2^{\RR^n}$ be a maximal monotone operator, then
 $$x=(I+\beta B)^{-1}(I-\beta A)(x),$$ if and only if, $0\in (A+B)(x)$.
\end{proposition}
\begin{proof}
See Proposition $3.13$ in \cite{PhD-E}.
\end{proof}

\noindent Now we define the so called Fej\'er convergence.
\begin{definition}
Let $S$ be a nonempty subset of $\RR^n$. The sequence $(x^k)_{k\in \NN}\subset \RR^n$ is said to be Fej\'er convergent to $S$, if and only if, for all $x\in S$ there exists $k_0\ge 0$, such that $\|x^{k+1}-x\| \le \|x^k - x\|$ for all $k\ge k_0$.
\end{definition}

This definition was introduced in \cite{browder} and have been
further elaborated in \cite{IST} and \cite{borw-baus}. A useful result on Fej\'er sequences is the following.

\begin{proposition}\label{punto}
If $(x^k)_{k\in \NN}$ is Fej\'er convergent to $S$, then:
\begin{enumerate}
\item the sequence $(x^k)_{k\in \NN}$ is bounded;
\item the sequence $(\|x^k-x\|)_{k\in \NN}$ is convergent for all $x\in S;$
\item if a cluster point $x^*$ belongs to $S$, then the sequence $(x^k)_{k\in \NN}$ converges to $x^*$.
\end{enumerate}
\end{proposition}

\begin{proof}
(i) and (ii) See Proposition $5.4$ in \cite{librobauch}. (iii) See Theorem $5.5$ in \cite{librobauch}.
\end{proof}

\section{The Algorithm}\label{section3}
Let $A_i:\dom(A_i)\subset\RR^n \rightarrow \RR^n$ and $B_i:\dom(B_i)\subset\RR^n\rightarrow 2^{\RR^n} $ be maximal monotone operators, with $A_i$ point-to-point and $B_i$ point-to-set, for all $i\in \mathbb{I}$. we assume that:
\begin{enumerate}[leftmargin=0.5in, label=({\bf A\arabic*})]

\item\label{a1} $dom (B_i)\subseteq dom (A_i)$, for all $i \in \mathbb{I}:=\{1,2,3, \cdots, m\}$ with $m\in \NN$.
\item \label{a2} $S_*\ne \emptyset$.
\item\label{a3} For each bounded subset $V \subset \cap_{i=1}^m dom(B_i)$ there exists $R>0$, such that $B_i(x)\cap B[0,R]\neq\emptyset$,  for all $x\in V$ and $i\in \mathbb{I}$.
\end{enumerate}
 Where $B[0,R]$ is the closed ball centered in $0$ and radius $R$. We emphasize that this assumption holds trivially if $dom(B_i)=\RR^n$ or $V\subset int(dom(B_i))$ or $B_i$ is the normal cone in any subset of $dom(B_i)$ for all $i\in\mathbb{I}$, i.e., in the application to system of variational inequality problem, this assumption is not necessary.

Choose any nonempty, closed and convex set, $X \subseteq \cap_{i\in \mathbb{I}}dom (B_i)$, satisfying  $X\cap S_*\ne \emptyset$, The explanation for the chosen of $X$ can be found in \cite{tseng, rei-yun, phdthesis}.
Let $(\beta_k)_{k=0}^{\infty}$ be a sequence such that $(\beta_k)_{k\in \NN}\subseteq [\check{\beta},\hat{\beta}] $ with $0<\check{\beta} \leq \hat{\beta}<\infty$, and $\theta, \delta\in(0,1)$, let $R>0$ like Assumption \ref{a3}. The algorithm is defined as follows:

\begin{center}
\fbox{\begin{minipage}[b]{\textwidth}
\begin{Calg}{A}\label{concep} Let $(\beta_k)_{k\in \NN}, \theta, \delta, R \mbox{ and } \mathbb{I}$ like above.
\begin{retraitsimple}
\item[] {\bf Step~0 (Initialization):} Take $x^0\in X$.

\item[] {\bf Step~1 (Iterative Step 1):} Given $x^k$, define $z_1^k:=x^k$. Begin the process:  for $i=1$ to $m$ do
\begin{equation}{\label{jota}}
J_i^{k}:=(I+\beta_{k}B_i)^{-1}(I-\beta_{k}A_i)(z_i^{k}).
\end{equation}
If $z_i^k=J_i^k$ put $i\in \mathbb{I}^*_k$ set $ z_{i+1}^k=z_i^k$ and goto {\bf Step 1}.

{\bf Stopping Criteria} If $\mathbb{I}_k^*=\mathbb{I}$ , then $x^k \in S_*$.
\item[] {\bf Step~1.1 (Inner Loop):} Begin the inner loop over $j$.
 Put $j=0$ and choose any \newline $u_{(j,i)}^{k}\in B_i\big(\theta^{j}J_i^k+(1-\theta^{j})z_i^k\big)\cap B[0,R]$. If
\begin{equation}\label{jk}
\Big \la A_i \big(\theta^{j}J_i^k+(1-\theta^{j})z^k_i\big)+u^{k}_{(j,i)}, z^k_i-J_i^k\Big \ra\geq \frac{\delta}{\beta_k}\|z_i^k -J_i^k\|^2,
\end{equation}
then $j_i(k):=j$ and stop.
Else, $j=j+1$.
 Define:
\begin{equation}{\label{alphak}}
\alpha_{k,i}:=\theta^{j_i(k)},
\end{equation}
\begin{equation}{\label{ubar}}
\bar{u}_i^k:=u^k_{j_i(k)}
\end{equation}
\begin{equation}{\label{xbar}}
\bar{x}_i^k:=\alpha_{k,i} J_i^k+(1-\alpha_{k,i})x^k
\end{equation}
\begin{equation}{\label{Fk}}
z_{i+1}^{k}=P_X\big(P_{H_i(\bar{x}_{i}^{k},\bar{u}_{i}^{k})}(z_{i}^{k})\big).
\end{equation}
\item[] {\bf Step~2 (Iterative Step 2):} Define:
\begin{equation}\label{xk1}
x^{k+1}:=z^{k}_{m+1},
\end{equation}
set $k=k+1$, empty $\mathbb{I}^*_k$ and goto {\bf Step 1}.
\end{retraitsimple}
\end{Calg}\end{minipage}}
\end{center}

where
\begin{equation}\label{H(x)}
H_i(x,u) := \big\{ y\in \RR^n :\la A_i(x)+u,y-x\ra\le 0\big \}
\end{equation}
This method combine the Alternating Projection Method, the Forward-Backward Method and the ideas of separating hyperplane.

\section{Convergence Analysis}\label{section4}
In this section we analyze the convergence of the algorithms presented in the previous section. First, we present some general properties as well as prove the well-definition of the algorithm.

\begin{lemma}\label{propseq}
For all $(x,u)\in Gr(B_i)$, $S_*^i\subseteq H_i(x,u)$, for all $i\in \mathbb{I}$. Therefore $S_* \subset H_i(x,u)$ for all $i \in \mathbb{I}$.
\end{lemma}

\begin{proof}
Take $x^{*}\in S_*^i$. Using the definition of the solution, there exists $v^{*}\in B_i(x^{*})$, such that $0=A_i(x^{*})+v^{*}$. By the monotonicity of $A_i+B_i$, we have
$$\la A_i(x)+u -(A_i(x^{*})+v^{*}), x-x^{*}\ra\ge 0, $$
 for all $(x,u)\in Gr(B_i)$.
Hence,
$$\la A_i(x)+u, x^{*}-x\ra \le 0$$
and by \eqref{H(x)}, $x^{*}\in H_i(x,u)$.
\end{proof}

\noindent From now on, $(x^k)_{k\in \NN}$ is the sequence generated by the algorithm.
\begin{proposition}\label{propdef}
 The algorithm  is well-defined.
\end{proposition}
\begin{proof}
 The proof of the well-definition of $j_i(k)$ is by contradiction. Fix $i\in \mathbb{I}\setminus \mathbb{I}_k^*$ and assume that for all $j\ge0$ having chosen $u_{(j,i)}^{k}\in B_i\big(\theta^j J_i^k+(1-\theta^j)z_i^k\big)\cap B[0,R]$,
\begin{equation*}
\Big\la A_i \big(\theta^{j}J_i^k+(1-\theta^{j})z_i^k\big)+u^{k}_{j}, z_i^k-J_i^k\Big\ra < \frac{\delta}{\beta_k}\|z_i^k - J_i^k\|^2.
\end{equation*}
Since the sequence $\{u^{k}_{(j,i)}\}_{j=0}^{\infty}$ is bounded, there exists a subsequence $\{u^{k}_{(\ell_j,i)}\}_{j=0}^{\infty}$ of $\{u^{k}_{(j,i)}\}_{j=0}^{\infty}$, which converges to an element $ u_i^k$ belonging to $B_i(z_i^k)$ by maximality. Taking the limit over the subsequence $\{\ell_j\}_{j\in \NN}$, we get
\begin{equation}{\label{lim}}
\big\la\beta_k A_i(z_i^k)+\beta_k u_i^k, z_i^k -J_i^k\big \ra \le \delta \|z_i^k - J_i^k\|^2.
\end{equation}
It follows from (\ref{jota}) that
\begin{equation*}{\label{res}}
 \beta_k A_i(z_i^k)=z_i^k-J_i^k-\beta_k v_i^k,
 \end{equation*}
 for some  $ v_i^{k}\in B_i(J_i^k)$.\\
Now, the above equality together with (\ref{lim}), lead to
\begin{equation*}
\|z_i^k - J_i^k\|^2\le\Big\la z_i^k-J_i^k-\beta_k v_i^k+\beta_k u_i^k, z_i^k -J_i^k\Big \ra \le \delta \|z_i^k - J_i^k\|^2,
\end{equation*}
using the monotonicity of $B_i$ for the first inequality. So,
$$(1-\delta)\|z_i^k - J_i^k\|^2\le 0,$$
which contradicts that $i\in \mathbb{I}\setminus \mathbb{I}_k^*$. Thus, the algorithm is well-defined.
\end{proof}

Finally, a useful algebraic property on the sequence generated by the algorithm, which is a direct consequence of the inner loop and \eqref{xbar}.
\begin{corollary}\label{coro}
Let $(x^k)_{k\in \NN}$, $(\beta_k)_{k\in \NN}$ and $(\alpha_{(k,i)})_{k\in \NN}$ be sequences generated by the algorithm. With $\delta$ and $\hat{\beta}$ as in the algorithm. Then,
\begin{equation}\label{desig-muy-usada}
\la A_i(\bar{x}_i^{k})+\bar{u}_i^{k},z_i^{k}-\bar{x}_i^{k} \ra  \ge\frac{\alpha_{k,i}\delta}{\hat{\beta}}\|z_i^{k}-J_i^k)\|^2\geq0,
\end{equation}
for all $k$.
\end{corollary}

\begin{proposition}\label{stop1}
If the algorithm stops, then $x^k\in S_*$.
\end{proposition}

\begin{proof}
If Stop Criteria is satisfied, then $\mathbb{I}_k^*=\mathbb{I}$ then, by Proposition \ref{parada} we have that $x^k \in S_*^i$ for all $i\in \mathbb{I}$ which imply that $x^k\in S_*$.
 \end{proof}

From now on assume that the algorithm generate an infinite sequence $(x^k)_{k\in\NN}$.
\begin{proposition}\label{prop2}
\begin{enumerate}
\item The sequence $(x^k)_{k\in \NN}$ is Fej\'er convergent to $S_*\cap X$.
\item The sequence $(x^k)_{k\in \NN}$ is bounded.
\item For all $x^*\in S_*\cap X$ we have $\lim_{k\to \infty}\|z_j^k-x^*\|^2$ exist for all $j\in \mathbb{I}$ and satisfy that  $\lim_{k\to \infty}\|z_j^k-x^*\|^2$=$\lim_{k\to \infty}\|z_i^k-x^*\|^2$ for all $i,j\in \mathbb{I}$.
\item $\lim_{k\to \infty}\|x^{x+1}-x^k\|^2=0$.
\end{enumerate}
\end{proposition}
\begin{proof}
\begin{enumerate}
\item Take $x^*\in S_*\cap X$.  Using \eqref{Fk}, \eqref{xk1}, Proposition \ref{proj}(i) and Lemma \ref{propseq}, we have
\begin{eqnarray}\label{fejer-des}\nonumber\|x^{k+1}-x^{*}\|^2&=&\|z_{m+1}^k-x^{*}\|^2= \|P_{X}(P_{H_m(\bar{x}_m^k,\bar{u}_m^k)}(z_m^k))-P_X(P_{H_m(\bar{x}_m^k,\bar{u}_m^k)}(x^{*}))\|^2\\&\leq& \|z_m^k-x^*\|^2 \leq \cdots \leq \|z_1^k-x^*\|^2=\|x^k-x^*\|^2.\end{eqnarray} So, $\|x^{k+1}-x^{*}\|\le \|x^k-x^*\|$.

\item Follows immediately from item (i).

\item Take $x^* \in S_*\cap X$. Using \eqref{fejer-des} yields for all $i \in \mathbb{I}$ that
\begin{equation}\label{ineq}
\|x^{k+1}-x^*\|^2\le \|z_i^k-x^*\|^2\leq \|x^{k}-x^{*}\|^2.
\end{equation}
Now using Proposition \ref{punto} and item (ii) taking limits over $k$ we have that $(\|z_i^k-x^*\|^2)_{k\in \NN}$ is convergent for the same limits that $(\|x^k-x^*\|^2)_{k\in\NN}$, independent of the $i\in \mathbb{I}$, obtaining the result.

\item Is a direct consequence of item (iii).
\end{enumerate}
\end{proof}
\begin{proposition}\label{cadai}
For all $i\in \mathbb{I}$ we have,
$$\lim_{k\to \infty}\la A_i(\bar{x}_i^k)+\bar{u}_i^k,z_i^k-\bar{x}_i^k \ra=0.$$
\end{proposition}
\begin{proof}
For all $i\in \mathbb{I}$. Using Proposition \ref{proj}(i) and \eqref{Fk} for all $x^*\in S_*\cap X$ we have
\begin{align}\label{todoi}
\|z_{i+1}^{k}-x^*\|^2 =&\|P_{X}(P_{H_i(\bar{x}_i^k,\bar{u}_i^k)}(z_i^k))-P_{X}(P_{H_i(\bar{x}_i^k,\bar{u}_i^k)}(x^*))\|^2\le \|P_{H_i(\bar{x}_i^k,\bar{u}_i^k)}(z_i^k)-P_{H_i(\bar{x}_i^k,\bar{u}_i^k)}(x^*)\|^2 \nonumber \\
\leq &\|z_i^k-x^*\|^2-\|P_{H_i(\bar{x}_i^k,\bar{u}_i^k)}(z_i^k)-z_i^k\|^2.
\end{align}
Now reordering \eqref{todoi}, we get
$$\|P_{H_i(\bar{x}_i^k,\bar{u}_i^k)}(z_i^k)-z_i^k\|^2 \leq \|z_{i}^k-x^*\|^2-\|z_{i+1}^{k}-x^*\|^2.$$
Using the fact that,
$$P_{H(\bar{x}_i^k,\bar{u}_i^k)}(z_i^k)=z_i^k-\frac{\la A_i(\bar{x}_i^k)+\bar{u}_i^k,z_i^k-\bar{x}_i^k  \ra}{\|A_i(\bar{x}_i^k)+\bar{u}_i^k\|^2}(A_i(\bar{x}_i^k)+\bar{u}_i^k),$$
and the previous equation, we have,
\begin{equation}\label{pasar al lim}
\frac{\big(\la A_i(\bar{x}_i^k)+\bar{u}_i^k,z_i^k-\bar{x}_i^k  \ra\big)^2}{\|A_i(\bar{x}_i^k)+\bar{u}_i^k\|^2}\leq \|z_i^k-x^*\|^2-\|z_{i+1}^{k}-x^*\|^2.
\end{equation}
By Proposition \ref{inversa} and the continuity of $A_i$ we have that $J_i$ is continuo, since $(x^k)_{k\in \NN}$ and $(\beta_k)_{k\in \NN}$ are bounded then $\{J_i^k\}_{k\in \NN}$, $)(z_i^k)_{k\in \NN}$ and $\{\bar{x}_i^k\}_{k\in \NN}$ are bounded, implying the boundedness of $\{\|A_i(\bar{x}_i^k)+\bar{u}_i^k\|\}_{k\in \NN}$ for all $i\in \mathbb{I}$.

\noindent Using Proposition \ref{prop2}(iii), the right side of \eqref{pasar al lim} goes to 0, when $k$ goes to $\infty$, establishing the result.
\end{proof}

\begin{proposition}\label{samecluster}
For all $i\in \mathbb{I}$ we have $\lim_{k \rightarrow \infty}\|z_{i+1}^k-z_i^k\|=0$.
\end{proposition}
\begin{proof}
By definition of $z_{i+1}^k$ and using that $z_i^k\in X$ for all $i\in \mathbb{I}$ and $k\in\NN$ we have that
\begin{equation}\label{zero}
\|z_{i+1}^k-z_i^k\| =\|P_X(P_{H_i(\bar{x}_i^k,\bar{u}_i^k)}(z_i^k))-P_X(z_i^k)\|\leq\|P_{H_i(\bar{x}_i^k,\bar{u}_i^k)}(z_i^k)-z_i^k\|.
\end{equation}
The right side of the equation \eqref{zero} go to zero by Proposition \ref{cadai}, then the result follow.
\end{proof}

A direct consequence of the previous proposition is that $\lim_{k\rightarrow} \infty\|x^{k+1}-x^k\|=0$, just summing and resting $z_i^k$ for $i=2,3, \cdots, m-1$ and using Cauchy-Swartz, we have the result. Other direct consequence of the Proposition \ref{samecluster} is that the sequences generated by the algorithm $(x^k)_{k\in\NN}$, $(z_i^k)_{k\in\NN}$ for each $i\in \mathbb{I}$ have the same clusters points.

\noindent Next we establish our main convergence result of the algorithm.
\begin{theorem}\label{teo1}
The sequence $(x^k)_{k\in \NN}$ converges to some element belonging to $S_*\cap X $.
 \end{theorem}
\begin{proof}
 Since  $(x^k)_{k\in \NN}$ is bounded then have cluster points, we claim that they belongs to  $S_*\cap X$, as every $x^k$ belong to $X$ by definition and $X$ is closed, then all clusters point of $(x^k)_{k\in\NN}$ belong to $X$. The sequence $(x^k)_{k\in\NN}$ is Fej\'er convergent to the set $S_*\cap X$, then by Proposition \ref{punto} (iii) the whole sequence will be convergent to this set. Let $(x^{j_k})_{k\in \NN}$ be a convergent subsequence of $(x^k)_{k\in \NN}$ such that, for all $i\in \mathbb{I}$ the sequences  $(z_i^{j_k})_{k\in\NN}$, $(\bar{x}_i^{j_k})_{k\in \NN}, (\bar{u}_i^{j_k})_{k\in \NN}, (\alpha_{j_k,i})_{k\in \NN}$ and $(\beta_{j_k})_{k\in \NN}$ are convergents, and as we see before as consequence of Proposition \ref{samecluster}, calling the limits of $(x^{j_k})_{k\in \NN}$ as $\tilde{x}$, we have $\lim _{k\to \infty}x^{j_k}=\lim _{k\to \infty}z_i^{j_k} =\tilde{x}$, for all $i\in \mathbb{I}$. \\
Using Proposition \ref{cadai} and taking limits in \eqref{desig-muy-usada} over the subsequence $(j_k)_{k\in \NN}$, we have for all $i\in \mathbb{I}$,
\begin{equation}\label{limite}
0=\lim_{k\to \infty}\la A_i(\bar{x}_i^{j_k})+\bar{u}_i^{j_k},z_i^{j_k}-\bar{x}_i^{j_k} \ra \ge \lim_{k\to \infty}  \frac{\alpha_{j_k,i}\delta}{\hat{\beta}}\|x^{j_k}-J_i^{j_k}\|^2\geq0.
\end{equation}
Therefore,
\begin{equation*}
\lim_{k\to \infty} \alpha_{j_k,i}\|z_i^{j_k}-J_i^{j_k}\|=0.
\end{equation*}
Now consider the two possible cases.

(a) First, assume that $\lim_{k\to \infty}\alpha_{j_k,i}\ne 0$, i.e., $\alpha_{j_k,i}\ge \bar{\alpha}$ for all $k$ and some  $\bar{\alpha}>0$. In view of (\ref{limite}),

    \begin{equation}\label{limcero}
    \lim_{k\to \infty}\|z_i^{j_k}-J_i^{j_k}\|=0.
    \end{equation}
    Since $J_i$ is continuous, by the continuity of $A_i$ and $(I+\beta_k B_i)^{-1}$ and by Proposition \ref{inversa}, \eqref{limcero} becomes
    \begin{equation*}
    \tilde{x}=J_i(\tilde{x},\tilde{\beta}):=(I+\tilde{\beta} B_i)^{-1}(I-\tilde{\beta} A_i)(\tilde{x}),
    \end{equation*}
    which implies that $\tilde{x}\in S_*^i$ for all $i\in \mathbb{I}$ using Proposition \ref{parada}. Then $\tilde{x}\in S_*$ establishing the claim.

 (b) On the other hand, if $\lim_{k\to \infty}\alpha_{j_k,i}=0$ then for $\theta \in (0,1)$ as in the algorithm, we have
    $$\lim_{k\to\infty}\frac{\alpha_{j_k,i}}{\theta}=0.$$
    Define
    $$y^{j_k}_i:=\frac{\alpha_{j_k,i}}{\theta}J_i^{j_k}+\Big(1-\frac{\alpha_{j_k,i}}{\theta}\Big)z_i^{j_k}.$$
    Then,
    \begin{equation}\label{ykgox}
    \lim_{k\to\infty}y_i^{j_k}=\tilde{x}.
    \end{equation}
    Using the definition of the $j_i(k)$ and \eqref{alphak}, we have that $y_i^{j_k}$ does not satisfy \eqref{jk} implying
    \begin{equation}\label{conse}
    \Big \la A_i (y^{j_k}_i)+u^{j_k}_{j_i(k)-1}, z_i^{j_k}-J_i^{j_k}\Big \ra <\frac{\delta}{\beta_{j_k}}\|z_i^{j_k} -J_i^{j_k}\|^2,
    \end{equation}

    for $u^{j_k}_{j(j_k)-1,i}\in B_i(y^{j_k}_i)$ and all $k\in \NN$ and $i\in \mathbb{I}$.\\
    Redefining the subsequence $\{j_k\}_{k\in \NN}$, if necessary, we may assume that  $\{u^{j_k}_{j(j_k)-1,i}\}_{k\in \NN}$ converges to $\tilde{u}_i$. By the maximality of $B_i$, $\tilde{u}_i$ belongs to $B_i(\tilde{x})$. Using the continuity of $J_i$, $(J_i^{j_k})_{k\in \NN}$ converges to $J_i(\tilde{x},\tilde{\beta})$ as defined in the first case. Using \eqref{ykgox} and taking limit in (\ref{conse}) over the subsequence $(j_k)_{k\in \NN}$, we have
    \begin{equation}\label{tiu}
    \Big\la A_i(\tilde{x}) + \tilde{u}_i, \tilde{x}- J_i(\tilde{x},\tilde{\beta})\Big\ra \le \frac{\delta}{\tilde{\beta}}\|\tilde{x}-J_i(\tilde{x},\tilde{\beta})\|^2.
    \end{equation}
    Using the definition of $J_i(\tilde{x},\tilde{\beta}):=(I+\tilde{\beta} B_i)^{-1}(I-\tilde{\beta} A_i)(\tilde{x})$ and multiplying by $\tilde{\beta}$ on both sides of (\ref{tiu}), we get
    \begin{equation*}
    \la \tilde{x}-J_i(\tilde{x},\tilde{\beta})-\tilde{\beta}\tilde{v}_i+\tilde{\beta}\tilde{u}_i, \tilde{x}-J_i(\tilde{x},\tilde{\beta})\ra \le \delta\| \tilde{x}-J_i(\tilde{x},\tilde{\beta})\|^2,
    \end{equation*}
    where $\tilde{v}_i\in B_i(J_i(\tilde{x},\tilde\beta))$. Applying the monotonicity of $B_i$, we obtain
    $$\| \tilde{x}-J_i(\tilde{x},\tilde{\beta})\|^2 \le \delta \| \tilde{x}-J_i(\tilde{x},\tilde{\beta})\|^2,$$
    implying that $\| \tilde{x}-J_i(\tilde{x},\tilde{\beta})\|\le 0$. Thus, $\tilde{x}=J_i(\tilde{x},\tilde{\beta})$ and hence, $\tilde{x}\in S_*^i$ for all $i\in \mathbb{I}$, thus $\tilde{x} \in S_*$ This prove the convergence of the whole sequence to the set set $S_*\cap X$.
\end{proof}

\section{Conclusions}
In this paper, we present an hybrid algorithm combining a variant of forward-backward splitting methods and the alternative projection method for solving a system o inclusion problems composed by the sum of two operators. A linesearch, for relax the hypothesis of Lipschitz continuity on forwards operators, have been proposed. The convergence analyze of the algorithm is proved. The results presented here, improve the previous in the literature by relaxing the hypothesis and the subproblems calculated here are computationally cheapest that knowing in the literature.

\section*{Acknowledgments}

The author was partially supported by CNPq grant 200427/2015-6. This
work was concluded while the author visiting the University of South
Australia, the author would like to thanks the great hospitality received.
\bigskip

\bibliographystyle{plain}

\end{document}